\documentclass[11pt]{amsart}
\usepackage{microtype}
\usepackage{amsmath,amssymb,amsthm,amsfonts}
\usepackage{mathrsfs}
\usepackage{tikz-cd}

\usepackage{graphicx}
\usepackage{caption,subcaption}

\usepackage[T1]{fontenc}
\usepackage[utf8]{inputenc}
\usepackage{slashed}
\usepackage{dsfont}

\usepackage{hyperref}
\usepackage{mathtools}
\mathtoolsset{showonlyrefs}

\usepackage{csquotes}
\usepackage[margin=0.9in]{geometry}

\DeclareMathOperator{\Rm}{R}
\DeclareMathOperator{\diverg}{div}

\newcommand{\dd}{\mathop{}\!\mathrm{d}}
\newcommand{\A}{\mathbb{A}}
\newcommand{\D}{\slashed{D}}
\newcommand{\p}{\partial}
\newcommand{\al}{\alpha}
\newcommand{\be}{\beta}
\newcommand{\na}{\nabla}
\newcommand{\R}{\mathbb{R}}
\newcommand{\C}{\mathbb{C}}
\newcommand{\pd}{\slashed{\partial}}
\newcommand{\MS}[1]{M^{#1}}
\newcommand{\vph}{\varphi}
\newcommand{\vep}{\varepsilon}

\newtheorem{thm}{Theorem}[section]
\newtheorem{lemma}[thm]{Lemma}
\newtheorem{Def}{Definition}[section]

\title[Coarse regularity]{Coarse Regularity of Solutions to a Nonlinear Sigma-model with~$L^p$~gravitino}

\begin{document}

\author{Jürgen Jost, Ruijun Wu, Miaomiao Zhu}

\address{Max Planck Institute for Mathematics in the Sciences\\Inselstr. 22--26\\D-04103 Leipzig, Germany}
	\email{jjost@mis.mpg.de}
	
\address{Max Planck Institute for Mathematics in the Sciences\\Inselstr. 22--26\\D-04103 Leipzig, Germany}
	\email{Ruijun.Wu@mis.mpg.de}

\address{School of Mathematical Sciences, Shanghai Jiao Tong University\\Dongchuan Road 800\\200240 Shanghai, P.R.China}
	\email{mizhu@sjtu.edu.cn}

\thanks{%
\emph{Acknowledgements:}
Ruijun Wu thanks the International Max Planck Research School Mathematics in the Sciences for financial support.
Miaomiao Zhu was supported in part by the National Natural Science Foundation of China (No. 11601325).
}

\date{\today}

\begin{abstract}
 The regularity of weak solutions of a two-dimensional nonlinear sigma model with coarse gravitino is shown. Here the gravitino is only assumed to be in $L^p$ for some $p>4$. The precise regularity results  depend on the value of $p$.
\end{abstract}


\maketitle

\section{Introduction}
The action functionals of the various models of quantum field theory yield many examples of beautiful variational problems. These problems are usually analytically very difficult, because they represent borderline cases, due to phenomena like conformal invariance. What makes them still tractable usually is their intricate algebraic structure resulting from the various symmetries of and the interactions between the various fields involved. Mathematically, often a geometric interpretation of these algebraic structures is possible. In any case, the analysis needs to use the special structure of the action functional. A well known instance is the theory of harmonic mappings from Riemann surfaces to Riemannian manifolds, which in the context of QFT arise from the action functional of the nonlinear sigma model, or the Polyakov action of string theory. Here, a particular skew symmetry of the nonlinear term in the Euler-Lagrange equations could be systematically exploited and generalized in the work of H\'elein, Rivi\`ere and Struwe, see \cite{helein1991,riviere2007conservation,riviere2010conformally,riviere2008partial}. This is also our starting point, both conceptually -- because we generalize the harmonic map problem --  and methodologically -- because we shall use their techniques. In fact, the action functional of the nonlinear sigma model and the Polyakov action of string theory constitute only the simplest of their kind. In more sophisticated models, other fields enter, in particular a spinor field. Also, when one investigates the harmonic action functional mathematically, naturally also another object enters, the metric $g$ or the conformal structure of the underlying Riemann surface, and for many purposes, not only the field, but also $g$ should be varied. Again, however, in the advanced QFT models, there arises another object, a kind of partner of the metric $g$, the gravitino $\chi$, also called the Rarita-Schwinger field. In harmonic map theory, or in related theories, like Teichm\"uller theory \`a la Ahlfors-Bers, one often needs to consider metrics $g$ that are not necessarily smooth, and this may lead to delicate regularity questions. Likewise, the gravitino is not necessarily smooth, and in this paper we address the related regularity questions

In fact, this article is a part of our systematic study of an action functional motivated from super string theory. Let us now describe its ingredients in more precise terms. They are  a map from an oriented Riemann surface to a compact Riemannian manifold and its super partner, a vector spinor, with the  Riemannian metric of the domain and its super partner, the gravitino, as parameters. This action functional is the two-dimensional nonlinear sigma model of quantum field theory, which has been studied for a long time both in physics and mathematics. Such models have been used in supersymmetric string theory since the 1970s, see e.g. \cite{deser1976complete, brink1976locally}.   We refer to \cite{deligne1999quantum, jost2009geometry, jost2014super} for more details about the  mathematical aspects.

In a recent work \cite{jost2016regularity}, a corresponding geometric model was set up and some analytical issues were studied. In contrast with the previous models which use anticommuting fields and which are therefore not directly amenable to the methods of geometric analysis, this model uses only commuting fields and thus is given within the context of Riemannian geometry. Though this approach makes the supersymmetries involved less transparent, it has the advantages that this model is closely related to  mathematically long-studied models such as harmonic maps and Dirac-harmonic maps and their various variants. In \cite{jost2016regularity}, a detailed setup for this two-dimensional nonlinear sigma model was developed. On this basis, now  the regularity issues can be investigated. The smoothness of  weak solutions of the Euler--Lagrange equations, with smooth Riemannian metric and gravitino, was obtained in \cite{jost2016regularity}.

The analysis of two-dimensional harmonic maps, and even more so, of Dirac-harmonic maps is quite subtle, because they constitute borderline cases for the regularity theory, with phenomena like bubbling. While the harmonic map case by now can be considered as well understood, and much is known about Dirac-harmonic maps, it turns out that major new  difficulties from the analytical perspective are caused by the gravitino, even if the gravitino is treated only as a parameter and not as a dependent variable in its own right. These difficulties arise from the way the gravitino is coupled with the spinor field in the action functional, see \eqref{action functional} below. These difficulties become even more severe if the gravitino in the model is not smooth. More precisely, we encounter the following question: \emph{what is the weakest possible assumption on the gravitino and under such an assumption how smooth will the critical points of the action functional  be? } Apparently in general we can no longer  expect $C^\infty$ differentiability, but one may still hope to improve the original regularity of the weak solutions.  Here we explore this  issue. We shall combine the regularity theory of \cite{riviere2007conservation,riviere2010conformally,riviere2008partial} with Morrey space theory and a subtle iteration argument to achieve what should be the optimal regularity results in our setting. 

Let us briefly recall the framework of the model in \cite{jost2016regularity}. For details we refer to that article and the references therein. Let $(M,g)$ be a closed Riemannian surface with a fixed spin structure and $(N,h)$ an $n$-dimensional closed Riemannian manifold. Let $S$ be a spinor bundle over $M$ associated to the given spin structure, which has real rank four. This spinor bundle is a Dirac bundle in the sense of \cite{lawson1989spin}. In particular, there is a canonical spin connection and spin metric which is a fiberwise real inner product\footnote{Note that in several previous works there was some ambiguity about the fiber metric, and here we take the real one rather than the Hermitian one, as clarified in \cite{jost2016regularity}.}. The Clifford multiplication by a tangent vector will be denoted by a dot when no confusion can arise. A gravitino is by definition a section $\chi$ of the vector bundle $S\otimes TM$. The Clifford multiplication gives rise to a map~$\delta_\gamma\colon S\otimes TM\to S$, which is given by multiplying the tangent vectors to the spinors. This map is linear and surjective, and moreover the following short exact sequence splits:
\begin{equation}
 0\to ker \to S\otimes TM \to S \to 0.
\end{equation}
The projection map to the kernel is denoted by $Q:S\otimes TM\to S\otimes TM$. In a local oriented orthonormal frame $(e_\al)$ of $M$, using the summation convention as always, this projection is given by
\begin{equation}
 Q(\chi^\al\otimes e_\al)=-\frac{1}{2} e_\be\cdot e_\al\cdot \chi^\be\otimes e_\al.
\end{equation}

Now let $\phi\colon M\to N$ be a map between Riemannian manifolds. One can consider the twisted spinor bundle $S\otimes \phi^*TN$. It is again a Riemannian vector bundle over $M$ and on it a twisted spin Dirac operator $\D$ is defined, which is essentially self-adjoint with respect to the inner product in $L^2(S\otimes \phi^*TN)$. Then the action functional is given by
\begin{equation}
 \label{action functional}
	\begin{split}
		\A(\phi, \psi;g, \chi)\coloneqq \int_M & |\dd \phi|_{T^*M\otimes \phi^*TN}^2
			+ \langle \psi, \D \psi \rangle_{S\otimes \phi^*TN}  \\
		&  -4\langle (\mathds{1}\otimes\phi_*)(Q\chi), \psi \rangle_{S\otimes\phi^*TN}
			-|Q\chi|^2_{S\otimes TM} |\psi|^2_{S\otimes \phi^*TN}
			-\frac{1}{6} \Rm^{N}(\psi) \dd vol_g,
	\end{split}
\end{equation}
where $\Rm^N$ is the pullback of the curvature of $N$ under $\phi$, and the curvature term in the action is defined, in a local coordinate $(y^i)$ of $N$ and with $\psi=\psi^i\otimes \phi^*(\frac{\p}{\p y^i})$, by
\begin{equation}
 -\frac{1}{6}\Rm^N(\psi)=-\frac{1}{6}\Rm^{N}_{ijkl}\left\langle \psi^i,\psi^k\right\rangle_S \left\langle\psi^j,\psi^l\right\rangle_S.
\end{equation}
One can easily check that this is independent of the choices of local orthonormal frames. Note that this action functional can actually be defined on the space
\begin{equation}
	\mathcal{X}^{1,2}_{1,4/3}(M,N)=\{(\phi,\psi)\big| \phi\in W^{1,2}(M,N), \psi\in\Gamma^{1,4/3}(S\otimes\phi^*TN)\}.
\end{equation}
Here by $\Gamma^{1,4/3}(S\otimes\phi^*TN)$ we mean the space of $W^{1,4/3}$ sections of the twisted spinor bundle \(S\otimes \phi^*TN\). It is then clear that an $L^4$ assumption on $\chi$ is sufficient to make the action functional well defined and finite valued.

We remark that the Lagrangian of the action appears in this form for  reasons of supersymmetry. Note that in the particular case where the gravitino vanishes, this reduces to the Dirac-harmonic map functional with curvature term introduced in \cite{chen2007} and further studied in e.g. \cite{branding2015some, branding2015energy, jost2015geometric}. If in addition, the curvature terms in the Lagrangian also vanish, this reduces to the Dirac-harmonic map functional introduced in \cite{chen2005regularity, chen2006dirac}, which is studied to a great extent in e.g.  \cite{jost2008riemannian, wang2009regularity, zhu2009regularity, chen2011boundary, sharp2016regularity}.

Taking a local oriented orthonormal frame $\{e_\al| \al=1,2\}$, the Euler--Lagrange equations are
\begin{equation}
 \begin{split}
 \tau(\phi)=&\frac{1}{2}\Rm^{\psi^*TN}(\psi, e_\al\cdot\psi)\phi_* e_\al-\frac{1}{12}S\na R(\psi)   \\
            &  -(\langle \na^S_{e_\be}(e_\al \cdot e_\be \cdot \chi^\al), \psi \rangle_S
	        	+ \langle e_\al \cdot e_\be \cdot \chi^\al, \na^{S\otimes\phi^*TN}_{e_\be} \psi \rangle_S),  \\
    \D\psi =& |Q\chi|^2\psi +\frac{1}{3}SR(\psi)+2(\mathds{1}\otimes \phi_*)Q\chi,
 \end{split}
\end{equation}
where we have used the following abbreviations:
\begin{equation}
 \begin{split}
  \Rm^{\phi^*TN}(\psi, e_\al\cdot\psi)\phi_* e_\al
   &=R^i_{jkl}\langle\phi^k,\na\phi^j\cdot\psi^l\rangle_S\;\phi^*(\frac{\p}{\p y^i}), \\
  SR(\psi)
   &=R^i_{jkl}(\phi)\langle\psi^l,\psi^j\rangle_S\psi^k\otimes\phi^*(\frac{\p}{\p y^i}), \\
  S\na R(\psi)
   &=\phi^*(\na^N R)_{ijkl}\langle\psi^i,\psi^k\rangle_S \langle\psi^j,\psi^l\rangle_S.
 \end{split}
\end{equation}

To deal with the regularity it is advantageous to embed $(N,h)$ isometrically into some Euclidean space, say $\R^K$, and transfer the various quantities on $N$ to their images/pushforwards of $\R^K$. Let $f\colon (N,h)\hookrightarrow (\R^K, \delta)$ be such a smooth isometric embedding with second fundamental form~$A$, and let $\phi'\equiv f\circ \phi\colon M\to f(N)\subset \R^K$ be the composed map and $\psi'\equiv f_{\#}\psi$ the pushforward vector spinor. It suffices to consider the regularity of $(\phi',\psi')$. Let $\{u^a|a=1,\cdots, K\}$ be the global coordinates of $\R^K$ and let $\nu_l, l=n+1,\cdots, K,$ be a local normal frame of the submanifold $f(N)$. Then $\phi'=(\phi'^1,\cdots,\phi'^K)$ can be viewed as a $\R^K$-valued function, and $\psi'=(\psi'^1,\cdots,\psi'^K)$, where each $\psi'^a$ is a spinor, satisfies
\begin{equation}
 \sum_a \nu_l^a\psi'^a=0, \quad n+1\le l\le K.
\end{equation}
Since regularity is a local issue, we may locate the problem on the unit disk $B_1\subset \R^2\cong \C$. Then the equations satisfied by $(\phi',\psi')$ are
\begin{equation}\label{equation for phi'}
 \Delta\phi'^a=(\omega_{\al}^{ab}+F^{ab}_\al+ T^{ab}_\al)\frac{\p\phi'^b}{\p x^\al}
               +Z^a_{ebcd}\left\langle\psi'^e,\psi'^c\right\rangle \langle\psi'^b,\psi'^d\rangle-\diverg V'^a,
\end{equation}
and
\begin{equation}\label{equation for psi'}
 \begin{split}
  \pd\psi'^a=& -\na\phi'^d\cdot\psi'^b\frac{\p\nu^b_l}{\p u^d}\nu^a_l(\phi')+|Q\chi|^2\psi'^a \\
           &+\frac{1}{3} \left(\langle\psi'^b,\psi'^d\rangle\psi'^c-\langle\psi'^c,\psi'^b\rangle\psi'^d\right)
             \frac{\p \nu^b_l}{\p u^d}\left(\frac{\p \nu_l}{\p u^c}\right)^{\top,a}
              -e_\al\cdot\na\phi'^a\cdot\chi^\al,
 \end{split}
\end{equation}
for each $a$, where $\Delta$ is the Euclidean Laplacian operator, $\pd$ the Euclidean Dirac operator, $\na$ the Euclidean gradient operator, the coefficients are written in the following antisymmetric form:
\begin{equation}
 \begin{split}
  \omega^{ab}_\al=&-\left(\frac{\p\phi'^c}{\p x^\al}\frac{\p\nu^a_l}{\p u^c}\nu^b_l
                          -\frac{\p\phi'^c}{\p x^\al}\frac{\p \nu^b_l}{\p u^c}\nu^a_l\right)=-\omega^{ba}_{\al}, \\
  F^{ab}_\al=& \left\langle\psi'^c, e_\al\cdot\psi'^d\right\rangle
              \left(\left(\frac{\p\nu_l}{\p u^d}\right)^{\top,b}\left(\frac{\p\nu_l}{\p u^c}\right)^{\top,a}
                    -\left(\frac{\p\nu_l}{\p u^d}\right)^{\top,a}\left(\frac{\p\nu_l}{\p u^c}\right)^{\top,b}\right)
             =-F^{ba}_\al, \\
  T^{ab}_\al=& \left(\frac{\p\nu^c_l}{\p u^b}V'^c_\al\nu^a_l-\frac{\p\nu^c_l}{\p u^a}V'^c_\al\nu^b_l\right)=-T^{ba}_\al,
 \end{split}
\end{equation}
while
\begin{equation}
\begin{split}
 V'^a=&\left\langle e_\al\cdot e_\be\cdot\chi^\al,\psi'^a\right\rangle_S e_\be, \\
 Z^a_{ebcd}=& -\frac{1}{6}\left(\langle\na A_{ik}, A_{jl}\rangle-\langle\na A_{il}, A_{jk}\rangle\right)
              \frac{\p y^i}{\p u^a}\frac{\p y^j}{\p u^b} \frac{\p y^k}{\p u^c} \frac{\p y^l}{\p u^d}.
\end{split}
\end{equation}
For a detailed clarification of these formulae we refer to \cite{jost2016regularity}. For the cases of the simpler models, namely Dirac-harmonic maps and Dirac-harmonic maps with curvature terms, see \cite{zhu2009regularity, chen2011boundary, sharp2016regularity, branding2015some, jost2015geometric}.

From an analytical point of view, we shall be considering the following more general system which contains the essential information: suppose that $\phi\in W^{1,2}(B_1,\R^K)$ and $\psi\in W^{1,4/3}(B_1,\R^4\otimes\R^K)$ satisfy
\begin{equation}\label{equation for phi}
 \Delta\phi^a=\Omega^{ab}\na\phi^b+Z^a|\psi|^4+\diverg V^a,
\end{equation}
and
\begin{equation}\label{equation for psi}
 \pd\psi^a=A^{ab}\psi^b+B^a,
\end{equation}
where $\Omega^{ab}\in L^2(B_1,\R^2)$, $Z^a\in L^\infty(B_1,\R)$ $A^{ab}\in L^2(B_1,\mathfrak{gl}(4K,\R))$ and
\begin{align}
 B^a=-e_\al\cdot\na\phi^a\cdot\chi^\al, & &V^a=\langle e_\al\cdot e_\be\cdot\chi^\al, \psi^a\rangle_S e_\be.
\end{align}
The important feature is that $\Omega$ is antisymmetric:
\begin{equation}
 \Omega^{ab}=-\Omega^{ba}.
\end{equation}

As it is a critical elliptic system, one expects some higher regularity of the solutions than what is assumed  apriori. Unfortunately, if $\chi$ is only assumed to be $L^4$, it is not yet clear how to achieve this. Therefore, we first  try to deal with an $L^p$ gravitino with $4<p\le \infty$. As we shall see in this article, this allows us to obtain some regularity results for the solutions of \eqref{equation for phi}-\eqref{equation for psi}. In this article we adopt the following convention.

\begin{Def}
 Let $1<p\le \infty$. We say that a measurable function $u\colon (X,\mu)\to \R$ is an almost $L^p$ function, denoted by $u\in L^{p-o}(X,\mu)$, if $u\in L^q(X,\mu)$ for any $1\le q<p$.
\end{Def}

For example, for a bounded domain $U\subset \R^2$ with smooth boundary (actually a Lipschitz boundary is enough), the Sobolev embedding theorem says
\begin{equation}
 W^{1,2}_0(U)\hookrightarrow L^{\infty-o}(U).
\end{equation}

Then we can state the first result.
\begin{thm}\label{result for abstract system}
 Let $4<p\le \infty$, and $\chi\in L^p(B_1)$. Let $\phi\in W^{1,2}(B_1,\R^K)$ and $\psi\in W^{1,4/3}(B_1,\R^4\otimes\R^K)$ be a weak solution of the system \eqref{equation for phi}-\eqref{equation for psi}. Then for $p_0=\frac{8}{5}+\frac{16}{15}\sqrt{6}\approx 4.2132\cdots$, the following holds:
 \begin{enumerate}
  \item[(1)] If $p>p_0$, then $\psi\in W^{1,p/2}_{loc}(B_1)$ and $\phi\in W^{1,p}_{loc}(B_1)$. Furthermore, there exists an $\vep=\vep(p)>0$ such that whenever $\|\phi\|_{W^{1,2}(B_1)}+\|\psi\|_{L^4(B_1)}\le \vep$, then for any $U\Subset B_1$,
             \begin{equation}
              \|\phi\|_{W^{1,p}(U)}+\|\psi\|_{W^{1,p/2}(U)}\le C\left(\|\phi\|_{W^{1,2}(B_1)}+\|\psi\|_{L^4(B_1)}\right)
             \end{equation}
             for some constant $C=C(p,U,\|Q\chi\|_{L^p(B_1)})>0$.
  \item[(2)] If $4<p\le p_0$, then there exist some $t_*=t_*(p)\in (4,\infty)$ and $q_*=q_*(p)\in (2,\frac{2p}{p-2})$ such that $\psi\in W^{1,\frac{2t_*}{2+t_*}-o}_{loc}(B_1)\hookrightarrow L^{t_*-o}_{loc}(B_1)$ and $\phi\in W^{1,q_*-o}(B_1)$. Furthermore, there exists an $\vep=\vep(p)>0$ such that whenever $\|\phi\|_{W^{1,2}(B_1)}+\|\psi\|_{L^4(B_1)}\le \vep$, then for any $U\Subset B_1$, and for any $t<t_*$ and $q<q_*$,
             \begin{equation}
              \|\phi\|_{W^{1,q}(U)}+\|\psi\|_{W^{1,\frac{2t}{2+t}}(U)}\le C\left(\|\phi\|_{W^{1,2}(B_1)}+\|\psi\|_{L^4(B_1)}\right)
             \end{equation}
             for some constant $C=C(p,q,t,U,\|Q\chi\|_{L^p(B_1)})>0$.
 \end{enumerate}
\end{thm}

The methods used here are quite typical in the analysis of geometric partial differential equations. As we are dealing with a  critical case for the Sobolev framework, we need  a little Morrey space theory. Then Rivière's regularity theory \cite{riviere2007conservation} and its extensions in e.g. \cite{riviere2008partial, riviere2010conformally, sharp2013decay, sharp2016regularity} enable us to utilize the antisymmetric structure of the equations for $\phi$ to improve the regularity. Using similar methods, regularity results  for weak solutions of the simpler models, namely Dirac-harmonic maps and Dirac-harmonic maps with curvature terms, are achieved in \cite{zhu2009regularity, chen2011boundary, wang2009regularity, branding2015some}. Here in this more general model, the structure of the system is even more complicated because of the divergence terms and the appearance of the gravitinos. In the present work, we obtain regularity results for  weak solutions for the case of coarse gravitinos.

With this result in hand, we turn to the system \eqref{equation for phi'}-\eqref{equation for psi'}. Now we may make use of the concrete expressions of the coefficients $\Omega^{ab}$'s and $A^{ab}$'s. That is, by Theorem \ref{result for abstract system}, $\phi'$ and $\psi'$ now have better integrability properties, hence so do the corresponding $\Omega^{ab}$'s and $A^{ab}$'s. A more precise analysis of these coefficients will then lead to  our main result.
\begin{thm}\label{result for the original system}
 Let $(\phi,\psi)\in \mathcal{X}^{1,2}_{1,4/3}(M,N)$ be a critical point of the action functional $\A$. Suppose the gravitino $\chi\in \Gamma^p(S\otimes TM)$ for some $p\in (4,\infty]$. Then $\phi\in W^{1,p}(M,N)$ and $\psi\in \Gamma^{1,p/2}(S\otimes\phi^*TN)$. In particular, they are H\"older continuous.
\end{thm}

The article is organized as follows. We first prepare some lemmata to handle the equations for $\psi$ and $\phi$ separately. Then we can use an iteration procedure to improve the regularity of the solutions to the system \eqref{equation for phi}-\eqref{equation for psi} step by step. One can directly start from the section of iterations, skipping the two sections in which the lemmata are prepared, and refer to it back when necessary. In the final section we analyze the original system \eqref{equation for phi'}-\eqref{equation for psi'} and prove Theorem \ref{result for the original system}. Unlike many other problems where the coupling of variables causes additional problems, here the coupling behavior helps to achieve our goals.

Before start we would like to express our thanks to  Marius Yamakou for producing the nice graphs with MATLAB.

%%%%%%%%%%%%%%%%%%%%%%%%%%%%%%%%%%%%%%%%%%%%%%%%%%%%%%%%%%%%%%%%%%%%%%%%%%%%%%%%%%%%%%%%%%%%%%%%%%%%%%%%%%%%
\section{Preparation Lemma for Spinor Components}

In this section, we first handle the more general Dirac type equation \eqref{equation for psi} for $\psi$, and show that the integrability of $\psi$ can be improved by using an estimate of the Riesz potentials. We start with a general dimension $m\ge 2$. Then the system \eqref{equation for psi} is located on $B_1(0)\subset \R^m$. Note that the Dirac operator here is Euclidean, which has an explicit fundamental solution (see e.g. \cite{ammann2003variational})
\begin{equation}
 \na G(x,y)=\frac{1}{m\omega_m} \frac{x-y}{|x-y|^m},
\end{equation}
where $G(x,y)$ is the fundamental solution for the Euclidean Laplacian operator on~$\R^m$ and~$\omega_m=|B_1(0)|$. Convolutions with $\na G$ can be controlled by the Riesz potential operator $I_1$ on~$\R^m$, which is defined on measurable functions via
\begin{equation}
 I_1(u)=\int_{\R^m} \frac{1}{|x-y|^{m-1}}u(y)\dd y.
\end{equation}
In \cite{adams1975note} a good estimate about the Riesz potential operator has been given. They combine in an indirect way to improve the integrability of solutions to \eqref{equation for psi}. Later we will concentrate on the two-dimensional case because we are mostly concerned with a Riemann surface.

Since there are different notational conventions on Morrey spaces, we need to make our conventions explicit.  Let $U\subset \R^m$ be a domain. For $0\le \lambda\le m$ and $1\le p<\infty$, the Morrey spaces on $U$ are defined as
\begin{equation}
 \MS{p,\lambda}(U)\coloneqq \left\{u\in L^p(U)\big| \|u\|_{\MS{p,\lambda}(U)} <\infty\right\},
\end{equation}
where
\begin{equation}
 \|u\|_{\MS{p,\lambda}(U)}\coloneqq \sup_{x\in U, r>0} \left(\frac{r^\lambda}{r^m}\int_{B_r(x)\cap U} |u(y)|^p\dd y\right)^{1/p}.
\end{equation}
One can verify that on a bounded domain $U$ the following inclusions hold: for any $p\in [1,\infty)$ and any $\lambda\in [1,m]$,
\begin{equation}
 L^\infty(U)=\MS{p,0}(U)\subset \MS{p,\lambda}(U)\subset \MS{p,m}(U)=L^p(U).
\end{equation}
In particular, when $m=2$, one has $\MS{4,2}(U)=L^4(U)$ and $\MS{2,2}(U)=L^2(U)$. Recall that~$\psi\in W^{1,4/3}(B_1)$, which means, by Sobolev embedding, that $\psi\in \MS{4,2}(B_1)$. For further properties of Morrey spaces we refer to \cite{giaquinta1983multiple}. In \cite{adams1975note} it is shown that for any $1<q<\lambda\le m$,
\begin{equation}\label{admas inequality}
 I_1\colon \MS{q,\lambda}(U)\to \MS{\frac{\lambda q}{\lambda-q},\lambda}(U)
\end{equation}
is a bounded linear operator.

\begin{lemma}\label{preparation lemma for psi}
 Let $m\ge 2$ and $\frac{4}{3}< s\leq2$. Suppose $\vph\in\MS{4,2}(B_1(0), \R^L\otimes \R^K)$ be a weak solution of the system
 \begin{equation}\label{equation for vph}
  \pd\vph^i=A^i_j \vph^j+B^i, \quad 1\le i\le K,
 \end{equation}
 where $B_1\equiv B_1(0)$ denotes the unit open ball in $\R^m$, $A\in \MS{2,2}(B_1,\mathfrak{gl}(\R^L\otimes\R^K))$ and $B\in \MS{s,2}(B_1, \R^L\otimes\R^K)$. Then there exists an $\vep_0=\vep_0(m,s,t)>0$ such that if
 \begin{equation}
  \|A\|_{\MS{2,2}(B_1)} \le \vep_0,
 \end{equation}
 then $\vph\in L^{t}_{loc}(B_1)$ for $4\le t< 4+\frac{4}{3}\frac{3s-4}{2-s}=\frac{8}{6-3s}$. Moreover, for any domain $U\Subset B_1$ (which means $U\subset\overline{U}\subset B_1$),
 \begin{equation}\label{local estimate for vph}
  \|\vph\|_{L^t(U)}\le C\left(\|\vph\|_{\MS{4,2}(B_1)}+\|B\|_{\MS{s,2}(B_1)}\right).
 \end{equation}
for some $C=C(U,m,s,t)>0$.
\end{lemma}

We remark that $\R^L\otimes\R^K$ represents the typical fiber of a twisted spinor bundle over the~$m$-dimensional unit ball $B_1$, which is trivial. By this lemma we see that, as long as $B$ in \eqref{equation for vph} has better regularity than $\MS{\frac{4}{3},2}$, the integrability of $\vph$ can be improved. Arguments of this type have been used to show the regularities for Dirac type equations in various contexts, see e.g. \cite{wang2010remark} in dimension $m\geq2$ and see e.g. \cite{sharp2016regularity, branding2015some} in dimension $m=2$. The above result improves that in Lemma 6.1 in \cite{jost2016regularity}, where the case of $s=2$ was done and we include the sketch of the proof here only for the convenience of readers.

\begin{proof}
 Since the case $s=2$ has been shown in \cite[Lemma 6.1]{jost2016regularity}, here we consider $s\in(\frac{4}{3},2)$.

 Let $x_0\in B_1$ and $0<R<1-|x_0|$. Take a cutoff function $\eta\in C^\infty_0(B_R(x_0))$ such that $0\le \eta\le 1$ and $\eta\equiv 1$ on $B_{R/2}(x_0)$. Then for each $1\le i\le K$, set
 \begin{equation}
  g^i(x)\coloneqq \na G *\left(\eta^2(A^i_j \vph^j+B^i)\right)(x)
  =\int_{\R^m} \frac{\p G(x,y)}{\p y^\al} \frac{\p}{\p y^\al} \cdot\left(\eta^2(A^i_j \vph^j+B^i)\right)(y)\dd y.
 \end{equation}
 Then
 \begin{equation}
  \pd g^i=\eta^2(A^i_j \vph^j+B^i),
 \end{equation}
 and in particular, $\pd g^i=\pd \vph^i$ on $B_{R/2}(x_0)$. Thus each
 \begin{equation}
  h^i\coloneqq \vph^i-g^i
 \end{equation}
 is harmonic in $B_{R/2}(x_0)$. Meanwhile $g^i$'s can also be controlled in the aforementioned way
 \begin{equation}
  |g^i|\le C\int_{\R^m} \frac{1}{|x-y|^{m-1}}\left|\eta^2(A^i_j \vph^j+B^i)\right|\dd y
        \le C I_1\left(\eta^2(A^i_j \vph^j+B^i)\right).
 \end{equation}
 Then, noting that
 \begin{equation}
  \|\eta^2 B\|_{\MS{\frac{4}{3},2}(\R^m)}
          \le \|B\|_{\MS{\frac{4}{3},2}(B_R(x_0))}
          \le {\omega_n}^{\frac{1}{s^*}} R^{\frac{2}{s^*}}\|B\|_{\MS{s,2}(B_R(x_0))}
 \end{equation}
 with $s^*>4$ satisfies $\frac{3}{4}=\frac{1}{s}+\frac{1}{s^*}$ and using \eqref{admas inequality} with $q=\frac{4}{3}$ and $\lambda=2$, one gets
 \begin{equation}
  \begin{split}
   \|g\|_{\MS{4,2}(\R^m)}
    &\le C\|I_1\left(\eta^2(A\vph+B)\right)\|_{\MS{4,2}(\R^m)}\le C\|\eta^2(A\vph+B)\|_{\MS{\frac{4}{3},2}(\R^m)} \\
    &\le C\|\eta A\|_{\MS{2,2}(\R^m)}\|\eta \vph\|_{\MS{4,2}(\R^m)} + C\|\eta^2 B\|_{\MS{\frac{4}{3},2}(\R^m)} \\
    &\le C\vep_0\|\vph\|_{\MS{4,2}(B_R(x_0))}+ C R^{\frac{2}{s^*}}|B|,
  \end{split}
 \end{equation}
where $|B|\equiv \|B\|_{\MS{s,2}(B_1)}$. As each $h^i$ is harmonic in $B_{R/2}(x_0)$, it follows that for any $\theta\in (0,1/6)$,
 \begin{equation}
  \|h^i\|_{\MS{4,2}(B_{\theta R}(x_0))} \le (4\theta)^{1/2}\|h^i\|_{\MS{4,2}(B_{R/2}(x_0))}.
 \end{equation}
 Hence, recalling $\vph=g+h$, one has
 \begin{equation}
  \|\vph\|_{\MS{4,2}(B_{\theta R}(x_0))}
  \le C_0 (\vep_0+\theta^{1/2})\|\vph\|_{\MS{4,2}(B_R(x_0))} + C_1|B|R^{2/s^*}.
 \end{equation}
 Fix any $\be\in (0,\frac{2}{s^*})$. Then there is a $\theta\in (0,\frac{1}{6})$ such that $2C_0\theta^{1/2}\le \theta^\be$. Then take $\vep_0>0$ small enough such that $2C_0\vep_0\le \theta^\be$. With such a choice one has
 \begin{equation}
  \|\vph\|_{\MS{4,2}(B_{\theta R}(x_0))} \le \theta^\be\|\vph\|_{\MS{4,2}(B_R(x_0))} +C_1|B|R^{2/s^*}.
 \end{equation}
 Then, by a standard iteration argument, one can show that, for any $0<r<R<1-|x_0|$, it always holds that
 \begin{equation}
  \|\vph\|_{\MS{4,2}(B_r(x_0))}
  \le \frac{1}{\theta^\be}\left(\frac{r}{R}\right)^\be\|\vph\|_{\MS{4,2}(B_R(x_0))}
      +\frac{C_1|B|}{\theta^{2\be}-\theta^{\frac{2}{s^*}+\be}} r^\be
 \end{equation}
 which in turn implies that
 \begin{equation}
  \left(\frac{1}{r^{m-2+4\be}}\int_{B_r(x_0)} |\vph|^4 \dd y\right)^{\frac{1}{4}}
  \le \frac{1}{(\theta R)^\be} \|\vph\|_{\MS{4,2}(B_1)} +\frac{C_1 |B|}{\theta^{2\be}-\theta^{\frac{2}{s^*}+\be}}.
 \end{equation}
 Therefore, taking $|x_0|<\frac{1}{4}$ and $R=\frac{1}{2}$, one sees $\vph\in \MS{4,2-4\be}(B_{1/4})$ for any $\be \in (0,\frac{2}{s^*})$ with
 \begin{equation}
  \|\vph\|_{\MS{4,2-4\be}(B_{1/4})}
  \le C(m,\be) \left(\|\vph\|_{\MS{4,2}(B_1)}+\|B\|_{\MS{s,2}(B_1)}\right).
 \end{equation}

 Next we improve the integrability. As before for any $x_1\in B_{1/4}$ and any $0<R<\frac{1}{4}-|x_1|$, take a cutoff function $\eta\in C^\infty_0(B_R(x_1))$ and define $g^i$ and $h^i$ in the same way. This time with~$q=\frac{4}{3}$  and~$\lambda=2-\frac{4\be}{3}$, one has
 \begin{equation}
  \begin{split}
   \|g\|_{\MS{\frac{4(3-2\be)}{3-6\be}, 2-\frac{4\be}{3}}(\R^m)}
   &\le C\|I_1\left(\eta^2(A\vph+B)\right)\|_{\MS{\frac{4(3-2\be)}{3-6\be}, 2-\frac{4\be}{3}}(\R^m)}
     \le C\|\eta^2(A\vph+B)\|_{\MS{\frac{4}{3},2-\frac{4\be}{3}}(\R^m)} \\
   &\le C\|\eta A\|_{\MS{2,2}(\R^m)} \|\eta\vph\|_{\MS{4,2-4\be}(\R^m)}
        +C\|\eta^2 B\|_{\MS{\frac{4}{3},2-\frac{4\be}{3}}(\R^m)}  \\
   &\le C \vep_0\|\vph\|_{\MS{4,2-4\be}(B_1)}+ C\|B\|_{\MS{s,2}(B_1)} R^{\frac{2}{s^*}-\be}.
  \end{split}
 \end{equation}
 Since the harmonic part $h$ is smooth in $B_{R/2}(x_1)$, it behaves nicely with respect to all Morrey norms in an interior domain. In particular one can get
 \begin{equation}
  \|h\|_{\MS{\frac{4(3-2\be)}{3-6\be}, 2-\frac{4\be}{3}}(B_{R/3}(x_1))}
  \le C \left(\|\vph\|_{\MS{4,2}(B_1)}+\|B\|_{\MS{s,2}(B_1)}\right).
 \end{equation}
 Therefore, $\vph=g+h$ can be estimated by
 \begin{equation}
  \|\vph\|_{\MS{\frac{4(3-2\be)}{3-6\be}, 2-\frac{4\be}{3}}(B_{1/16})}
  \le C(n,\be)\left(\|\vph\|_{\MS{4,2}(B_1)}+\|B\|_{\MS{s,2}(B_1)}\right).
 \end{equation}
 Recall that $\be$ can be arbitrarily chosen in $(0,\frac{2}{s^*})$. Since
 \begin{equation}
  \lim_{\be\nearrow\frac{2}{s^*}} \frac{4(3-2\be)}{3-6\be}=4+\frac{4}{3}\frac{8}{s^*-4}=4+\frac{4}{3}\frac{3s-4}{2-s},
 \end{equation}
 and
 \begin{equation}
  \MS{\frac{4(3-2\be)}{3-6\be}, 2-\frac{4\be}{3}}(B_{1/16})\hookrightarrow
     L^{\frac{4(3-2\be)}{3-6\be}}(B_{1/16}),
 \end{equation}
 one concludes that $\vph\in L^t(B_{1/16})$ for any $t<4+\frac{4}{3}\frac{3s-4}{2-s}$. The desired estimate \eqref{local estimate for vph} also follows in a standard way. For details of the above argument one can consult \cite{jost2016regularity}. This completes the proof.

\end{proof}

In our case, we have $m=2$, since $B^a=-e_\al\cdot \na\phi^a\cdot \chi^\al\in L^{\frac{2p}{2+p}}(B_1)$, so $s_0=\frac{2p}{2+p}\in (\frac{4}{3},2)$. By Lemma \ref{preparation lemma for psi} we immediately get $\psi\in L^{t_1-o}_{loc}(B_1)$ with $t_1\equiv \frac{8}{6-3s_0}=\frac{2}{3}(p+2)$. Note that $t_1>4$ whenever $p>4$, so the integrability of $\psi$ is improved, although only by a little. Moreover, for any $t<t_1$ and any $U\Subset B_1$, we have the estimate
\begin{equation}\label{first improvement for psi}
 \|\psi\|_{L^t(U)} \le C(U,p,t)\left(\|\psi\|_{L^4(B_1)}+\|\na \phi\|_{L^2(B_1)}\|\chi\|_{L^p(B_1)}\right).
\end{equation}
We point out that the above argument doesn't work when $p=4$. This is a crucial issue.

%%%%%%%%%%%%%%%%%%%%%%%%%%%%%%%%%%%%%%%%%%%%%%%%%%%%%%%%%%%%%%%%%%%%%%%%%%%%%%%%%%%%%%%%%%%%%%%%%%%%%%%%%%%%%%%%%%%%%
\section{Preparation Lemma for Map Components}

Now the equations \eqref{equation for phi} for $\phi$ are almost away from being critical, and we will show that the map has better regularity than $W^{1,2}(B_1,\R^K)$. Note that~$\Omega^{ab}\na\phi^b\in L^1(B_1)$ and~$\diverg V^a\in W^{-1,2}(B_1)$, and both of them may cause trouble. The following lemma, which is a combination of Campanato regularity theory and Rivière's regularity theory, will be useful for handling these problems.

\begin{lemma}\label{preparation lemma for phi}
 Let $p,t \in (4,\infty]$. Suppose that $u=(u^1,\cdots, u^K)\in W^{1,2}(B_1,\R^K)$ solves the following system
 \begin{equation}
  -\Delta u^a=\Omega^{ab}\na u^b +f^a+\diverg V^a, \quad 1\le a\le K,
 \end{equation}
 where $\Omega\in L^{2}(B_1, \mathfrak{so}(K)\otimes\R^2)$, $f\in L^{t/4}(B_1,\R^K)$ and $V\in L^{\frac{pt}{p+t}}(B_1,\R^K\otimes \R^2)$. Then there exists an $\vep_1=\vep_1(p,t,K)>0$ such that if $\|\Omega\|_{L^2(B_1)}\le \vep_1$, then~$u\in W^{1,\frac{2\sigma}{2-\sigma}}_{loc}(B_1,\R^K)$, where~$\sigma=\frac{2pt}{2(p+t)+pt}\wedge \frac{t}{4}$, and for any $U\Subset B_1$,
 \begin{equation}\label{local estimate for u}
  \|u\|_{W^{1,\frac{2\sigma}{2-\sigma}}(u)} \le C\left(\|u\|_{L^2(B_1)} +\|f\|_{L^{\frac{t}{4}}(B_1)} +\|V\|_{L^{\frac{pt}{p+t}}(B_1)}\right)
 \end{equation}
 for some constant $C=C(U,p,t,K)>0$.
\end{lemma}

\emph{Remark.}
Note that here $B_1$ is the unit open disk in $\R^2$. For two real numbers $x,y\in \R$, we have used the notation
\begin{equation}
 x\wedge y =\min \{x,y\}.
\end{equation}
Moreover, when $t=\infty$, then $\sigma=\frac{2p}{2+p}$, and the lemma says that $u\in W^{1,p}_{loc}(B_1)$.

\begin{proof}
 Decompose $u=v+w$ where $v\in W^{1,2}_0 (B_1)$ is the solution of
 \begin{equation}
  \begin{cases}
   -\Delta v=\diverg V, & \textnormal{in } B_1 \\
   v=0, & \textnormal{on } \p B_1.
  \end{cases}
 \end{equation}
 The existence and uniqueness are ensured by \cite[Chap. 8]{chen1998elliptic}. By Campanato space theory, we know that~$\na v\in L^{\frac{pt}{p+t}}(B_1)$ and
 \begin{equation}
  \|\na v\|_{L^{\frac{pt}{p+t}}(B_1)}\le C\|V\|_{L^{\frac{pt}{p+t}}(B_1)}
 \end{equation}
 for some $C=C(\frac{pt}{p+t})$. Note that $\frac{pt}{p+t}>2$ since $p,t>4$. It then follows from Poincar\'e's inequality that
 \begin{equation}
  \|v\|_{W^{1,\frac{pt}{p+t}}(B_1)}\le C\|V\|_{L^{\frac{pt}{p+t}}(B_1)}.
 \end{equation}
 On the other hand, $w\in W^{1,2}(B_1)$ satisfies
 \begin{equation}
  \begin{cases}
   -\Delta w=\Omega \na w+\Omega\na v+f, & \textnormal{in } B_1,\\
   w=u, &\textnormal{on } \p B_1.
  \end{cases}
 \end{equation}
 Now we know that $\Omega\na v\in L^{\frac{2pt}{2(p+t)+pt}}(B_1)$ and $f\in L^{\frac{t}{4}}(B_1)$. Set $\sigma$ to be the smaller one of the two, that is,
 \begin{equation}
  \sigma\coloneqq \frac{2pt}{2(p+t)+pt}\wedge \frac{t}{4}
  =\begin{cases}
    \frac{2pt}{2(p+t)+pt}, & \textnormal{if  } \frac{6p}{2+p}\le t; \\
    \frac{t}{4} , & \textnormal{if  } \frac{6p}{2+p}\ge t.
   \end{cases}
 \end{equation}
 Then $1<\sigma<2$ and $\Omega\na v+f\in L^\sigma (B_1)$. At this stage we can use \cite[Theorem 1.1]{sharp2013decay} to conclude that as long as $\|\Omega\|_{L^2}\le \vep_1(p,t,K)$ is small enough, one has $w\in W^{2,\sigma}_{loc}(B_1)$ and for any $U\Subset B_1$,
 \begin{equation}
  \begin{split}
   \|w\|_{W^{2,\sigma}(U)}
   &\le C\left(\|w\|_{L^1(B_1)}+\|\Omega\na v+f\|_{L^\sigma(B_1)}\right) \\
   &\le C\left(\|u\|_{L^2(B_1)}+\|v\|_{W^{1,\frac{pt}{p+t}}(B_1)}+\|f\|_{L^{\frac{t}{4}}(B_1)}\right) \\
   &\le C\left(\|u\|_{L^2(B_1)}+\|V\|_{L^{\frac{pt}{p+t}}(B_1)}+\|f\|_{L^{\frac{t}{4}}(B_1)}\right),
  \end{split}
 \end{equation}
 for some $C=C(U,p,t,K)>0$. The Sobolev embedding says that
 \begin{equation}
  W^{2,\sigma}(U)\hookrightarrow W^{1,\frac{2\sigma}{2-\sigma}}(U)
  =\begin{cases}
    W^{1,\frac{pt}{p+t}}(U), & \textnormal{if  } \frac{6p}{2+p}\le t; \\
    W^{1,\frac{2t}{8-t}}(U), & \textnormal{if  } \frac{6p}{2+p}\ge t.
   \end{cases}
 \end{equation}
 Therefore, if $\frac{6p}{p+t}\le t$, then $v,w\in W^{1,\frac{pt}{p+t}}_{loc}(B_1)$, and so is $u=v+w$; and if $\frac{6p}{2+p}\ge t$, since in this case $\frac{2t}{8-t}\le \frac{pt}{p+t}$, we then have
 \begin{equation}
  u=v+w\in W^{1,\frac{2t}{8-t}}_{loc}(B_1).
 \end{equation}
 The desired local estimate \eqref{local estimate for u} follows directly. The proof is thus finished.

\end{proof}

Again note that
\begin{equation}
 \frac{2\sigma}{2-\sigma}>2
\end{equation}
as long as $p,t> 4$. We will apply it to the equation \eqref{equation for phi} with $\psi\in L^{t_1-o}_{loc}(B_1)$ where $t_1=\frac{2}{3}(p+2)$ as in the previous section. Then we conclude that $\na\phi\in L^{q_1-o}_{loc}(B_1)$ with
\begin{equation}
 q_1=\frac{pt_1}{p+t_1}\wedge \frac{2t_1}{8-t_1}=\frac{2p(p+2)}{5p+4}\in (2,p).
\end{equation}
Moreover, for any $U\Subset B_1$ and any $q< q_1$, we have the estimate
\begin{equation}\label{first improvement for phi}
 \|\phi\|_{W^{1,q}(U)}\le C(U,p,t,q)\left(\|\phi\|_{W^{1,2}(B_1)}+\|\psi\|^4_{L^t(U')}+\|\chi\|_{L^p(B_1)}\|\psi\|_{L^{t}(U')}\right),
\end{equation}
for some $t<t_1$, where $U\Subset U'\Subset B_1$.

%%%%%%%%%%%%%%%%%%%%%%%%%%%%%%%%%%%%%%%%%%%%%%%%%%%%%%%%%%%%%%%%%%%%%%%%%%%%%%%%%%%%%%%%%%%%%%%%%%%%%%%%%%%%%%%%%%%%%%%
\section{Improvement of Regularity by an Iteration Procedure}

In this section we prove Theorem \ref{result for abstract system}, and in the end we give two examples of different values of $p$ and different terminating values $q_*$.
\begin{proof}[Proof of Theorem \ref{result for abstract system}]
Consider a solution $(\phi,\psi)$ to \eqref{equation for phi}-\eqref{equation for psi}. As we have seen, after applying Lemma \ref{preparation lemma for psi} and Lemma \ref{preparation lemma for phi} once, one has
\begin{align}
 \psi\in L^{t_1-o}_{loc}(B_1)\cap L^4(B_1) \subsetneq L^4(B_1), & &\na\phi\in L^{q_1-o}_{loc}(B_1)\cap L^2(B_1)\subsetneq L^2(B_1).
\end{align}
Next we use an iteration argument to improve the regularities. As aforementioned, since there are some nonsmooth coefficients, one should not expect that this procedure goes to infinity. Actually it terminates at a certain point, as shown below. It may be reasonable to expect that $\phi\in W^{1,p}_{loc}(B_1)$ and $\psi\in W^{1,\frac{p}{2}}_{loc}(B_1)$. But we will see that for the system \eqref{equation for phi}-\eqref{equation for psi}, this is not always the case.

Before dealing with the general solutions, let's consider some particular cases.

First note that, once $\psi$ is shown to be in $L^{\infty-o}_{loc}(B_1)$ and $\na\phi\in L^{p-o}_{loc}(B_1)$, then the standard elliptic theory applied to \eqref{equation for psi} immediately implies
\begin{equation}
 \psi \in W^{1,\frac{p}{2}-o}_{loc}(B_1)\hookrightarrow C^0(\mathrm{\mathrm{int}} B_1)\hookrightarrow L^{\infty}_{loc}(B_1),
\end{equation}
where $\mathrm{int} B_1$ denotes the interior of the unit disk. It follows from the equations that $\na\phi\in L^p_{loc}(B_1)$ and thus~$\psi\in W^{1,\frac{p}{2}}_{loc}(B_1)$. Since the gravitino $\chi$ is involved in the divergence term, one cannot expect more.

Second, when $p=\infty$, the situation is almost trivial. Actually, now~$B^a=-e_\al\cdot \na\phi^a\cdot \chi^\al\in L^2(B_1)$ for each~$a$. From Lemma 6.1 in \cite{jost2016regularity} it follows that $\psi\in L^{\infty-o}_{loc}(B_1)$. Then applying Lemma \ref{preparation lemma for phi} we get $\phi\in W^{1,p-o}_{loc}(B_1)$. This returns to the situation above, and also finishes the proof for the case $p=\infty$.

Thus in the following we may assume $4<p<\infty$. We describe the abstract procedure by a recursive algorithm:
\begin{enumerate}
 \item[\textcircled{1}] Suppose it has been shown that $\psi\in L^t_{loc}(B_1)$ and $\na\phi\in L^q_{loc}(B_1)$ for some $t>4$ and~$q>2$.
 \item[\textcircled{2}] Then $B\in L^s_{loc}(B_1)$ with $s=s(q)=\frac{pq}{p+q}>\frac{4}{3}$.

                        If $s\ge 2$, then as before we immediately get $\psi\in L^{\infty-o}_{loc}(B_1)$ and $\na\phi\in L^{p-o}_{loc}(B_1)$. The desired result follows. Thus we may take $q<\frac{2p}{p-2}\equiv Q_0(p)$ in \textcircled{1} so that $s<2$.
 \item[\textcircled{3}] By Lemma \ref{preparation lemma for psi}, $\psi\in L^{T(q)-o}_{loc}(B_1)$ with
                        \begin{equation}
                         T\equiv T(q)=\frac{8}{6-3s(q)}=\frac{8(p+q)}{6p+6q-3pq}\in (4,\infty).
                        \end{equation}
 \item[\textcircled{4}] To determine the value of $\sigma$, we need to compare
                        \begin{equation}
                         \frac{T}{4}=\frac{2(p+q)}{6p+6q-3pq}
                        \end{equation}
                        and
                        \begin{equation}
                         \frac{2pT}{2(p+T)+pT}=\frac{8p(p+q)}{(-3p^2+10p+8)q+(10p^2+8p)}.
                        \end{equation}
                        A simple calculation shows that
                        \begin{equation}
                         \frac{T}{4}\ge \frac{2pT}{2(p+T)+pT}\Leftrightarrow q\ge \frac{14p^2-8p}{9p^2-14p+8}.
                        \end{equation}
                        Since $q>2$ while $\frac{14p^2-8p}{9p^2-14p+8}<2$ (since $p>4$ by assumption), the value of $\sigma$ is determined by
                        \begin{equation}
                         \sigma=\frac{2pT}{2(p+T)+pT}\wedge\frac{T}{4}=\frac{2pT}{2(p+T)+pT}
                         =\frac{8p(p+q)}{(-3p^2+10p+8)q+(10p^2+8p)}.
                        \end{equation}
                        For $q\in (2,\frac{2p}{p-2})$, $\sigma$ lies in the interval
                        \begin{equation}
                         \left(\frac{2p(p+2)}{p^2+7p+4}, \frac{2p}{p+2}\right),
                        \end{equation}
                        which is a proper subinterval of $(1,2)$. In particular, $\sigma<2$ and
                        \begin{equation}
                         \frac{2\sigma}{2-\sigma}=\frac{pT}{p+T}=\frac{8p(p+q)}{(-3p^2+6p+8)q+(6p^2+8p)}\eqqcolon Q(q)\equiv Q.
                        \end{equation}
 \item[\textcircled{5}] Lemma \ref{preparation lemma for phi} then shows that $\na\phi\in L^{Q-o}_{loc}(B_1)$.
 \item[\textcircled{6}] Compare the value of $q$ and $Q(q)$.

                        Case 1: $q<Q(q)<Q_0=\frac{2p}{p-2}$. Then go to \textcircled{1} with $\psi\in L^{T(q)-o}_{loc}(B_1)$ and $\na\phi\in L^{Q(q)-o}_{loc}(B_1)$, and then go through the procedure again.

                        Case 2: $Q(q)\ge Q_0$. Then $B\in L^2_{loc}(B_1)$. The desired result is obtained as before.

                        Case 3: $Q(q)\le q$. Then this procedure also terminates, with $t_*=T(q)$ and $q_*=Q(q)$ in the statement of Theorem \ref{result for abstract system}.

\end{enumerate}

Next we analyze the limiting behavior of such an iteration. It turns out that this is determined by $p$.

As indicated in step \textcircled{6}, we need to analyze the value of $Q(q)$. Consider the equation $Q(q)=q$, which is equivalent to
\begin{equation}
 (-3p^2+6p+8)q^2+6p^2q-8p^2=0.
\end{equation}
The discriminant is
\begin{equation}
 \begin{split}
   \Delta&=\left(6p^2\right)^2-4\left(-3p^2+6p+8\right)\times \left(-8p^2\right) \\
         &=4p^2\left(-15p^2+48p+64\right) \\
         &=4p^2\left[-15\left(p-\frac{8}{15}\right)^2+\frac{512}{5}\right].
 \end{split}
\end{equation}
Thus for $p>4$,
\begin{equation}
 \begin{cases}
  \Delta\ge 0, &\textnormal{if  } 4<p\le \frac{8}{15} \left(3+2\sqrt{6}\right), \\
  \Delta<0,    &\textnormal{if  } p>\frac{8}{15} \left(3+2\sqrt{6}\right),
 \end{cases}
\end{equation}
where $\frac{8}{15} \left(3+2\sqrt{6}\right)\approx 4.2132\cdots$ and we denote this number by $p_0$.

Even if $Q(q)=q$ has a solution, we still need to know whether the solution lies in the interval~$\left(2,Q_0(p)\right)$, where $Q_0(p)=\frac{2p}{p-2}$. This is actually the case, since the solutions are explicitly given by
\begin{equation}
 q_{\pm} =\frac{3p^2\pm p\sqrt{-15p^2+48p+64}}{3p^2-6p-8}.
\end{equation}
One can check that $q_{\pm}$ are always smaller than $Q_0(p)$ for $p>4$. Figure 1 shows the relation of $q_{\pm}$ and $Q_0$.

\begin{figure}[h!]
 \includegraphics[width=0.55\textwidth]{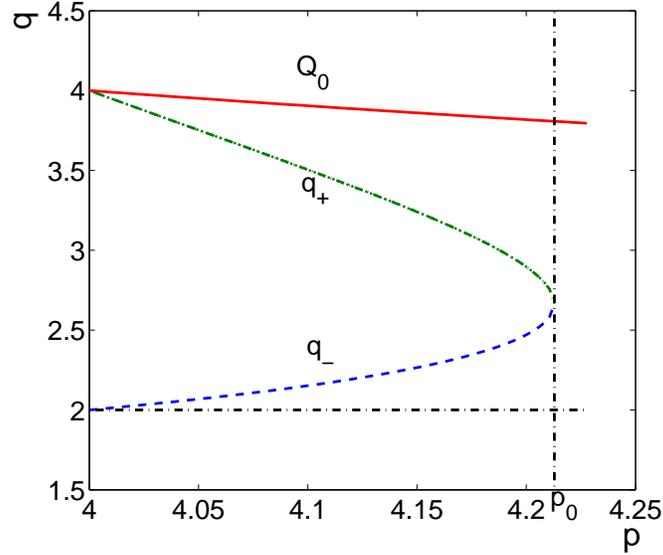}
 \caption{Comparison of $q_{\pm}$ and $Q_0$.}
\end{figure}

Thus the improvement will not work at $q_*=q_-(p)$ for $p\le p_0$. The corresponding $t_*$ is given by~$T(q_*)$. On the other hand if $p>p_0$, then one can easily get the regularity improved to the expected level.

The desired estimates follows from an iterated combination of \eqref{first improvement for psi} and \eqref{first improvement for phi}. The proof of Theorem \ref{result for abstract system} is completed.

\end{proof}

We remark that
\begin{equation}
 \frac{2t_*}{2+t_*}>\frac{\frac{p}{2}t_*}{\frac{p}{2}+t_*}=\frac{pq_*}{p+q_*},
\end{equation}
which prevents us from further improvements.

Finally we give two graphs to explain how the procedure works for both a large $p$ $(p=5)$ and a relatively small $p$ $(p=4.15)$.
\begin{figure}[!h]
 \centering
 \begin{subfigure}[b]{0.4\textwidth}
  \includegraphics[width=\textwidth]{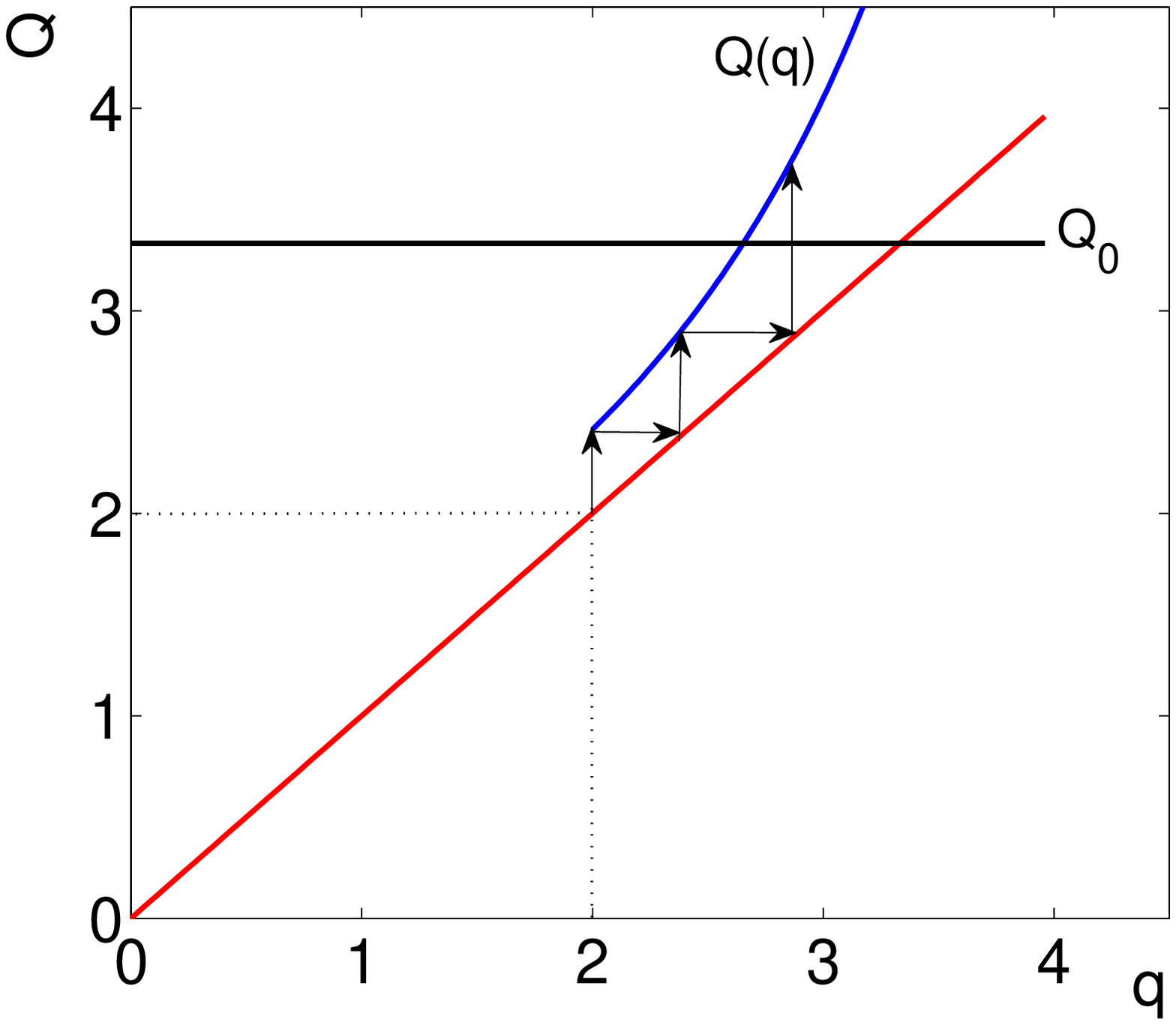}
  \caption{Figure 2 (p=5)}
 \end{subfigure}
 \begin{subfigure}[b]{0.415\textwidth}
  \includegraphics[width=\textwidth]{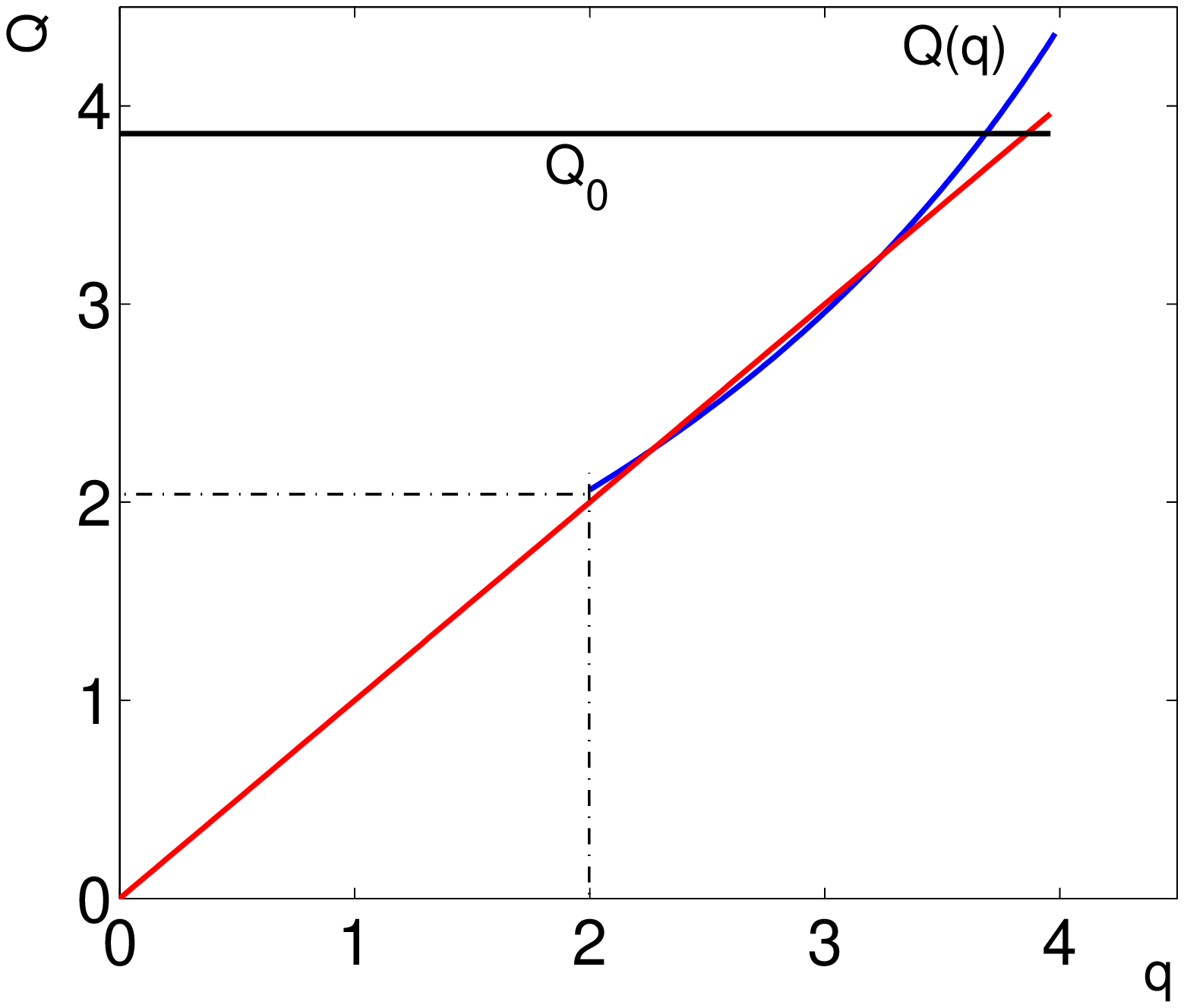}
  \caption{Figure 3 (p=4.15)}
 \end{subfigure}
\end{figure}

Here the horizontal lines stand for the barrier $Q_0(p)=\frac{2p}{p-2}$.

%%%%%%%%%%%%%%%%%%%%%%%%%%%%%%%%%%%%%%%%%%%%%%%%%%%%%%%%%%%%%%%%%%%%%%%%%%%%%%%%%%%%%%%%%%%%%%%%%%%%%%%%%%%%%%%
\section{Regularity of the Critical Points of the Action Functional}

We can now turn to the regularity of the critical points of the action functional \eqref{action functional}, or equivalently the solutions of the Euler--Lagrange equations \eqref{equation for phi'}-\eqref{equation for psi'}.
In contrast to Theorem \ref{result for abstract system}, the solutions of the Euler--Lagrange equations have the expected regularities, due to the structure of the equations.

\begin{proof}[Proof of Theorem \ref{result for the original system}]
Let $(\phi',\psi')$ be a solution to \eqref{equation for phi'}-\eqref{equation for psi'}. To prove Theorem \ref{result for the original system}, it suffices to consider the case where $4<p\le p_0$. Recall that Theorem \ref{result for abstract system} already gives $\psi'\in L^{t_*-o}_{loc}(B_1)$ and $\na\phi'\in L^{q_*-o}_{loc}(B_1)$ with
\begin{equation}
 q_*=\frac{3p^2-p\sqrt{-15p^2+48p+64}}{3p^2-6p-8}
\end{equation}
and
\begin{equation}
 t_*=T(q_*)=\frac{8p+8q_*}{6p+6q_*-3pq_*}.
\end{equation}
They are compared as
\begin{equation}
 4<p<2q_*<t_*.
\end{equation}
One should also note the following equalities
\begin{align}\label{limit equality}
 \frac{\frac{p}{2}t_*}{\frac{p}{2}+t_*}=\frac{pq_*}{p+q_*}, & & q_*=\frac{pt_*}{p+t_*}.
\end{align}
The regularity of $(\phi',\psi')$ is improved as follows. Set $t_0=t_*$ and $q_0=q_*$. We will temporarily use the notation
\begin{equation}
 L^q_{loc}(B_1)\cdot L^r_{loc}(B_1)\equiv
 \left\{ u\cdot v\big| u\in L^q_{loc}(B_1), v\in L^r_{loc}(B_1) \right\},
\end{equation}
for any $q,r\in [1,\infty]$. By H\"older inequality,
\begin{equation}
 L^q_{loc}(B_1)\cdot L^r_{loc}(B_1)\subset L^{\frac{qr}{q+r}}_{loc}(B_1).
\end{equation}
We may suppress the domain $B_1$ whenever it is clear.

First consider $\psi'$. Note that the coefficients $A^{ab}$'s are actually bad terms in the sense that
\begin{equation}
 A^{ab}\in L^{q_0}_{loc} \cap L^{\frac{t_0}{2}}_{loc} \cap L^{\frac{p}{2}}_{loc}=L^{\frac{p}{2}}_{loc},
\end{equation}
that is, it cannot be improved, due to the appearance of $|Q\chi|^2$ in $A^{ab}$. Thus by \eqref{equation for psi'} and thanks to \eqref{limit equality},
\begin{equation}
 \pd\psi'\in \left( L^{\frac{p}{2}}_{loc} \cdot L^{t_0-o}_{loc} \right) \bigcap \left( L^{q_0-o}_{loc} \cdot L^p\right)
 = L^{\frac{p}{2}}_{loc} \cdot L^{t_0-o}_{loc}
 = L^{q_0-o}_{loc} \cdot L^p.
\end{equation}
 It follows that
 \begin{equation}
  \psi'\in W^{1,\frac{p q_0}{p+q_0}-o}_{loc}(B_1)\hookrightarrow L^{t_1-o}_{loc}(B_1),
 \end{equation}
 with
 \begin{equation}
  t_1=\frac{2\cdot \frac{\frac{p}{2}t_0}{\frac{p}{2}+t_0}}{2-\frac{\frac{p}{2}t_0}{\frac{p}{2}+t_0}}
     =\frac{p}{p-(\frac{p}{2}-2)t_0}t_0>t_0,
 \end{equation}
 and hence
 \begin{equation}
  \frac{1}{t_1}=\frac{1}{t_0}-(\frac{1}{2}-\frac{2}{p})<\frac{1}{t_0}.
 \end{equation}
 On the other hand,
 \begin{equation}
  t_1=\frac{2\cdot \frac{pq_0}{p+q_0} }{2- \frac{pq_0}{p+q_0} }=\frac{2pq_0}{2p+2q_0-pq_0},
 \end{equation}
 from this it directly follows that
 \begin{equation}
  \frac{1}{t_1}=\frac{1}{q_0}+\frac{1}{p}-\frac{1}{2}.
 \end{equation}

 Next we turn to $\phi'$. As $t_1>t_0>\frac{6p}{2+p}$, by Lemma \ref{preparation lemma for phi}, we have
 \begin{equation}
  \na\phi'\in L^{q_1-o}_{loc}(B_1),
 \end{equation}
 with
 \begin{equation}
  q_1=\frac{pt_1}{p+t_1}, \quad \frac{1}{q_1}=\frac{1}{p}+\frac{1}{t_1}=\frac{1}{q_0}-(\frac{1}{2}-\frac{2}{p}).
 \end{equation}
 Note that this implies
 \begin{equation}
  \frac{1}{q_1}+\frac{1}{p}=\frac{2}{p}+\frac{1}{t_1}.
 \end{equation}

 Finally, by repeating such a procedure, we conclude that for $k\ge 1$,
  \begin{equation}
   \frac{1}{t_k}=\frac{1}{t_0}-k(\frac{1}{2}-\frac{2}{p}),
   \quad
   \frac{1}{q_k}=\frac{1}{q_0}-k(\frac{1}{2}-\frac{2}{p}).
  \end{equation}
  Therefore, after finitely many steps we are led to
  \begin{equation}
   \psi'\in W^{1,\frac{p}{2}}_{loc}(B_1), \quad \phi'\in W^{1,p}_{loc}(B_1).
  \end{equation}
  The conclusion of Theorem \ref{result for the original system} then follows.

\end{proof}

\end{document}